\documentclass[12pt]{article}

\usepackage[pdftex]{graphicx}
\usepackage{amsmath,amssymb,amsthm}
\usepackage{amsrefs}
\usepackage{enumitem}
\usepackage{xcolor}
\usepackage{tikz}             
\usepackage{hyperref}
\usepackage{lineno}
\usepackage[margin=1in]{geometry}

\numberwithin{equation}{section}

\hypersetup{colorlinks = false, linkcolor = blue, urlcolor = blue, citecolor = blue}

\usepackage{lineno}
\modulolinenumbers[5]

\theoremstyle{plain}
\newtheorem{theorem}{Theorem}[section]

\newtheorem{lemma}[theorem]{Lemma}

\newtheorem{corollary}[theorem]{Corollary}

\theoremstyle{definition}

\theoremstyle{remark}

\usepackage{authblk}

\begin{document}

\title{Two dimensional anisotropic mean curvature flow with\\contact angle condition}
\date{}

\author{Can Cui, Nung Kwan Yip \\
Department of Mathematics, Purdue University, West Lafayette, 47907 \\
cui147@purdue.edu, yipn@purdue.edu
}



\maketitle

\begin{abstract}
In this paper, we study surfaces which evolve by anisotropic mean curvature flow with contact angle boundary condition over a strictly convex domain in $\mathbb{R}^2$. We establish a prior gradient estimate for smooth solutions to this boundary value problem. The same approach can also handle Dirichlet boundary condition in $\mathbb{R}^n$, $n\geq 2$. For both problems, we prove that the solutions converge to one that is translation invariant in time.\\

\end{abstract}

\section{Introduction}
The evolution of an $n$-dimensional hypersurface $M(t) \subseteq \mathbb{R}^{n+1},\ t\geq 0$ by its mean curvature is given by the equation
\begin{equation} \label{eq:mcf_para}
\frac{\partial x}{\partial t}=\Vec{H}(x,t)
\end{equation}
where $x\in M(t)$, and $\Vec{H}(x,t)$ is the mean curvature vector at $(x,t)$. 
Equation \eqref{eq:mcf_para} is also the gradient flow of the surface area. In this sense, 
the mean curvature vector is the first variation of the surface area functional. 
This will be made more precise in the anisotropic case.
Mean curvature flow can be applied in various fields. It is considered in modeling the behavior of the interfaces in materials and the formation of microstructures in materials science, see \cite{Mu} and \cite{Ta}. It can also help enhancing features and removing noise from images in image processing and smoothing and simplifying complex meshes in computer graphics \cite{AGLM} and \cite{desbrun1999implicit}.

In this paper, we will concentrate on the case $M(t)$ is written as the graph of a function $u(\cdot ,t)$ over the $\mathbb{R}^n$-plane, i.e. 
$M(t) = \{ (x,u(x,t)) \in \mathbb{R}^n \times \mathbb{R} \}$. In this case, 
the unit upward normal for the graph of $u$ and its scalar mean curvature $H$ are given by
\begin{equation}
    \nu=\frac{(-Du,1)}{\sqrt{1+|Du|^2}}\quad\text{and}\quad
    H = \frac{1}{\sqrt{1+|Du|^2}} \left( \delta_{ij} -\frac{u_{x_i} u_{x_j}}{1+|Du|^2} \right) u_{x_i x_j}.
\end{equation}
where we use the Einstein convention of summing over repeated indices 
$1\leq i,j\leq n$.
As the normal velocity of the graph is given by $\frac{u_t}{\sqrt{1+|Du|^2}}$,  the evolution equation \eqref{eq:mcf_para} can be written in the following (non-parametric) form:
\begin{equation} \label{eq:mcf_nonpara}
    u_t =  \left( \delta_{ij} -\frac{u_{x_i} u_{x_j}}{1+|Du|^2} \right) u_{x_i x_j}.
\end{equation}

In \cite{EH1}, Ecker and Huisken proved that under Lipschitz initial data with linear growth, equation \eqref{eq:mcf_nonpara} has a smooth solution for all times. In \cite{EH2}, they relaxed the initial data to be just locally Lipschitz continuous. In \cite{Hu1}, Huisken considered \eqref{eq:mcf_nonpara} over a bounded domain $\Omega\subseteq\mathbb{R}^n$ with prescribed contact angle condition at the boundary:
\begin{equation} \label{eq:iso_ang_pro}
    \left\{ \begin{array}{cllcl}
            u_t & = & \left( \delta_{ij} -\frac{u_{x_i} u_{x_j}}{1+|Du|^2} \right) u_{x_i x_j}   \quad &  \text{in} & \ Q_T = \Omega \times [0,T], \\
            D_N u & = & - \sqrt{1+|Du|^2} \cos \theta  \quad &  \text{on} & \  \Gamma_T = \partial\Omega \times [0,T],  \\
         u(\cdot,0) & = & u_0(\cdot)  \quad \ & \text{in} & \ \  \Omega_0 = \Omega \times \{0\},
    \end{array}  \right.
\end{equation}
where the contact angle $\theta$ is described as the angle between the unit upward normal $\nu$ to the graph of $u$ and the unit inner normal $N$ to $\partial\Omega$ in 
the $\mathbb{R}^{n}$-plane, or equivalently the angle between the tangent plane of $u$ and the vertical direction -- see Figure 1, left. In \cite{Hu1}, Huisken only considered the case $\theta=\frac{\pi}{2}$, which reduces \eqref{eq:iso_ang_pro} to a homogeneous Neumann problem. He proved that the solution to \eqref{eq:iso_ang_pro} asymptotically converges to a constant function. In \cite{AW1}, Altschuler and Wu considered \eqref{eq:iso_ang_pro} with $n=1$ and a constant $\theta$ between $0$ and $\pi$. They proved that the solution converges to a translating solution.  This result also works for a class of quasilinear equations. Later, in \cite{AW2}, they proved the same convergence result over a strictly convex domain in $\mathbb{R}^{2}$ under the assumption that the derivative of $\theta$ along the boundary is bounded by the curvature of $\partial\Omega$. In physical settings, $\theta$ is the angle between the liquid (graph of $u$) and solid phases at their interface (see Figure 1, right) and hence describes the scenario that a liquid intersects with a solid phase -- see \cite{finn1986equilibrium} for physical background. 

\begin{figure}[h]
\centering
\begin{minipage}{0.35\linewidth}
\includegraphics[width=1\linewidth]{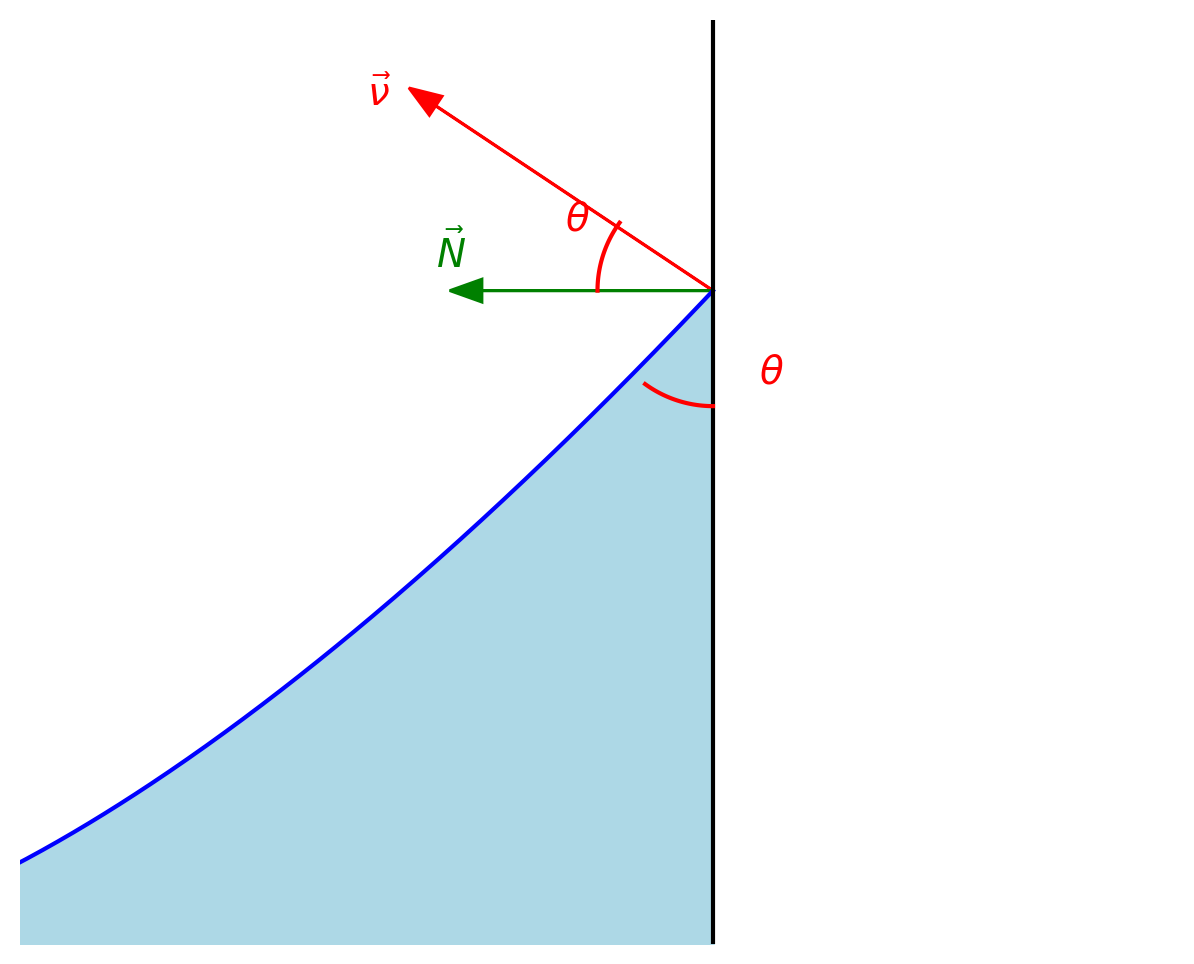}\\
\end{minipage}\hspace{20pt}
\begin{minipage}{0.45\linewidth}
\includegraphics[width=0.8\linewidth]{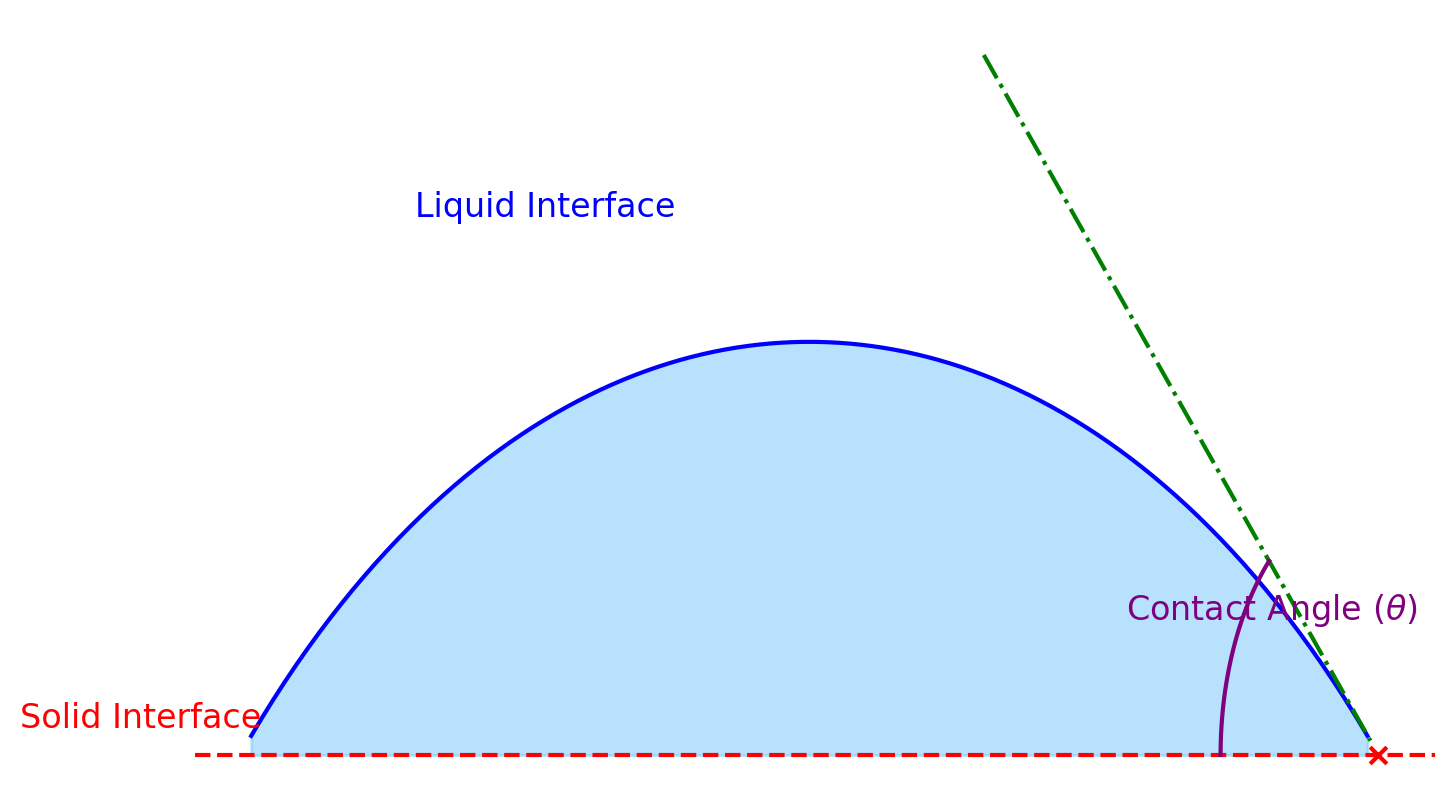}\\
\end{minipage}\hfill
\caption{  Left: contact angle $\theta$, Right: liquid contact with solid.}
\end{figure}

In this paper, we are interested in extending \eqref{eq:iso_ang_pro} to the anisotropic case. Compared to the isotropic equation, anisotropic mean curvature flow has several applications where directional dependence plays a crucial role. In crystalline materials, the energy associated with the surface often varies by direction due to atomic arrangement. Anisotropic mean curvature flow can model shape evolution of crystal surfaces more accurately than isotropic mean curvature flow by accounting for this directional dependency. Isotropic flow would tend to smooth surfaces uniformly, but anisotropic flow can capture faceted structures where certain directions are favored, leading to realistic simulations of crystal shapes with sharp edges and specific orientations.  We refer to \cite{taylor1, Ta} for further information on 
the use of mean curvature flows in the context of materials science.

Now we introduce the concept of anisotropy. Consider a positive function $F$ defined on the unit sphere $\mathbb{S}^{n}$ 
which is extended to $\mathbb{R}^{n+1}$ by homogeneity of degree one:
\[
F(\lambda\hat{p}) = \lambda F(\hat{p}),\quad\text{for any $\lambda > 0$ and $\hat{p}\in\mathbb{S}^n$}.
\]
Then the anisotropic analogue of equation \eqref{eq:mcf_nonpara} is given by
\begin{equation} \label{eq:aniso_mcf}
    u_t = \sqrt{1+|Du|^2} D^2_{p_i p_j} F (Du,-1) u_{x_i x_j}.
\end{equation}
(Note that in the above, $D^2_{p_ip_j}F$ refers to the second derivatives of $F$ with respect to the first $n$ variables with $1\leq i,j\leq n$.)
In order for the above equation to be parabolic, $F$ is required to be convex on $\mathbb{R}^{n+1}$.
We remark that equation \eqref{eq:mcf_nonpara} can in fact be put into the same form as \eqref{eq:aniso_mcf} if $F$ is 
simply given by $\overline{F}(p) = |p|$. Similar to the isotropic mean curvature flow \eqref{eq:mcf_nonpara}, \eqref{eq:aniso_mcf} arises as a gradient flow for the following surface energy functional: 
 \begin{equation} \label{eq:sur_energy}
    E(u) = \int F(-\nu)\sqrt{1+|Du|^2} dx
    \,\,\left(=
    \int F(Du, -1)dx
    \right).
\end{equation}
With this, we compute the variation of $E$ with respect to $u$:
\begin{eqnarray*}
\left.\frac{d}{dt}E(u+s\varphi)\right|_{s=0}
= \left.\frac{d}{ds}\int F(Du + sD\varphi, -1)dx\right|_{s=0}
= \int -D^2_{p_ip_j}F(Du,-1)u_{x_ix_j}\varphi\,dx.
\end{eqnarray*}
From the above, we define the first variation of $E$ as 
\[
\partial E = D^2_{p_i p_j} F (Du,-1) u_{x_i x_j}
\]
which is naturally called the anisotropic mean curvature.
In fact, in this paper, we can consider the following more general version of \eqref{eq:aniso_mcf}
\begin{equation} \label{eq:ansio_mob_mcf}
    u_t = \sqrt{1+|D u|^2}\,G(-\nu)D^2_{p_i p_j} F (Du,-1) u_{x_i x_j},
\end{equation}
by incorporating a positive mobility function $G(\cdot):\mathbb{S}^n\longrightarrow\mathbb{R}_+$. (We also extend $G$ to be a homogeneous degree one function defined on $\mathbb{R}^{n+1}$: $G(\lambda \hat{p}) = \lambda G(\hat{p})$
for any $\lambda \geq 0$ and $\hat{p}\in \mathbb{S}^{n}$.)

Here we mentioned some relevant results for anisotropic mean curvature flows. Given a convex, homogeneous function $F$ of degree one, in \cite{An1}, Andrews proved that the solution to the volume-preserving anisotropic mean curvature flow stays smooth and convex for all times, and converges to the Wulff shape corresponding to $F$ as $t\longrightarrow\infty$. Furthermore, in \cite{Cl}, under some assumption guaranteeing that $F$ does not deviate too much from the isotropic function $\overline{F}$, Clutterbuck proved an interior gradient estimate. Recently, in \cite{CKN}, Cesaroni, Kr{\"o}ner and Novaga construct soliton solutions as an entire graph assuming cylindrical anisotropies and mobilities -- see equation (4.1) in their work. 

Following the works \cite{Hu1}, \cite{AW1} and \cite{AW2}, in this paper, we study the contact angle boundary value problem:
\begin{equation} \label{eq:ang_pro}
    \left\{ \begin{array}{cllcl}
            u_t  &=& G(Du,-1) \  D^2_{p_i p_j} F (Du,-1) u_{x_i x_j}   \quad &  \text{in} & \ Q_T = \Omega \times [0,T], \\
            D_N u &=& - \sqrt{1+|Du|^2} \cos \theta  \quad &  \text{on} & \ \Gamma_T = \partial \Omega \times [0,T],  \\
         u(\cdot,0) &=& u_0(\cdot)  \quad \ & \text{in} & \ \  \Omega_0 = \Omega \times \{0\},
    \end{array}  \right.
\end{equation}
where $\Omega$ is a bounded and strictly convex domain in $\mathbb{R}^{2}$. 
Hence we generalize the result in \cite{AW2} to \eqref{eq:ang_pro}. In Section 2, we will present proper conditions  on the anisotropic function $F$ and its derivatives to guarantee that it does not deviate too much from the isotropic function $|p|$. 

We introduce our main results in Section 3. We establish a priori gradient estimate to the solution to \eqref{eq:ang_pro}, which implies uniform ellipticity of the operator. Existence and uniqueness of the solution then follow from the standard theory of quasilinear parabolic equations. The long time behavior of the solution to \eqref{eq:ang_pro} is related to the solution of the following elliptic problem: 
\begin{equation} \label{eq:ell_ang_pro}
\left\{ \begin{array}{cllcl}
     \lambda & = & G (Dw,-1) \ D^2_{p_i p_j} F (Dw,-1) w_{x_i x_j}   \quad & \text{in} & \ \Omega,  \\
     D_N w & = & - \sqrt{1+|Dw|^2} \cos \theta \quad & \text{on} & \ \partial \Omega. 
\end{array} \right.
\end{equation}
where the solution is the pair $(u,\lambda)$ with a constant $\lambda$ to be determined.
It turns out that the techniques to establish gradient estimate to \eqref{eq:ang_pro} is also applicable to \eqref{eq:ell_ang_pro}. This gradient estimate then implies the existence of $\lambda$. Uniqueness of $\lambda$ is derived from the boundary condition. As a result, the function $w$ leads to the translating solution $\Tilde{w}(x,t) = w(x) + \lambda t$ to \eqref{eq:ang_pro} with speed $\lambda$. Finally, an application of the strong maximum principle implies that solutions to \eqref{eq:ang_pro} converges to $\Tilde{w}$ (up to a constant).

Unfortunately this technique does not apply when $n>2$. This is due to the presence of mixed second order tangential derivatives along the boundary in higher dimensions, as we will remark in Section 3. Nevertheless, it can be applied to Dirichlet boundary problem in arbitrary dimension $n$. In this case, as shown in Section 4, the convexity condition of the domain boundary can be reduced to weighted mean convexity which is related to the anisotropic function. 
This condition, regarded as a replacement of the boundary curvature condition considered in \cite[Section 9]{Se} and \cite[Section 14.3]{GT}, is compatible with the results from the theory of quasilinear parabolic equations. Similar to Section 3, 
we also prove results for the corresponding elliptic version.

In recent years, the results in \cite{AW2} about isotropic mean curvature flows were generalized to arbitrary dimension $n$ -- see \cite{MWW} for Neumann boundary condition and \cite{GMWW} for contact angle condition. In an upcoming work \cite{CY2}, we are also able to 
modify their techniques to the anisotropic case.

\section{Preliminary Information}\label{sec:prelim}
Here we provide some preliminary computation useful for the rest of this paper.

First we give formulas for the derivatives of the function $|p|$ and our anisotropic function $F$ which is assumed to be convex and homogeneous of degree one. Note that
\begin{equation}
D_{p_i} |p|=\frac{p_i}{|p|},\quad\text{and}\quad
D^2_{p_i p_j} |p|=\frac{1}{|p|}(\delta_{ij}-\frac{p_i p_j}{|p|^2}).
\end{equation}
Upon introducing $\xi=\xi(p) := \frac{p}{|p|}$, we have
\begin{eqnarray}
F(p)&=&|p|F(\xi),\\
D_{p_i} F(p)&=&\frac{p_i}{|p|} F(\xi) + |p| D_{p_i} F(\xi),\\
D^2_{p_i p_j} F(p)&=&\frac{1}{|p|}(\delta_{ij}-\frac{p_i p_j}{|p|^2}) F(\xi) + \frac{p_i}{|p|}  D_{p_j} F(\xi) + \frac{p_j}{|p|}  D_{p_i} F(\xi) + |p| D^2_{p_i p_j} F(\xi).
\end{eqnarray}
Note that in the above, the indices $i,j$ range from $1$ to $n+1$.

The following assumptions are given to quantify the deviations of the mobility and anisotropic functions 
$G$ and $F$ from their isotropic versions, $\overline{G}(\xi)\equiv 1$ and $\overline{F}(\xi) \equiv 1$:
\begin{enumerate}
\item For $G(\xi)$, there exist positive constants $g_0$ and $G_0$, such that
\begin{equation} \label{eq:cond_mob}
    g_0 \leq G(\xi) \leq G_0.
\end{equation}
\item For $F(\xi)$, there exist positive constants $m_0, \ M_0, \ m_1, \ m_2$, such that for any $1\leq i,j\leq n$, 
\begin{equation} \label{eq:cond_aniso}
\begin{split}
    & m_0  \leq F(\xi) \leq M_0, \\
    & |D_{p_i} F(\xi)|  < \frac{m_1}{|p|}, \\
    & \left\| \left[ D^2_{p_i p_j} F(\xi) \right]\right\|  < \frac{m_2}{|p|^2},
\end{split}
\end{equation}
\end{enumerate}
where in the above, $\|M\|=\sqrt{\text{tr}(M^TM)}$ refers to the Frobenious norm of a matrix.
In Sections 3 and 4, we will give further conditions on these constants. 
As a reference, note that in the case of isotropic function $\overline{F}(p) = |p|$, we have $m_0=M_0=1$ and $m_1 =m_2=0$.

For convenience, we introduce the following notations which are used frequently later:
\begin{eqnarray*}
    v&=&\sqrt{1+|Du|^2},\\
    A \,(=A(Du)) = \big(a^{ij}\big),\,\,\,a^{ij}&=&G(Du,-1) \ D^2_{p_i p_j} F(Du,-1),\\
    \triangle_A&=&a^{ij}D_{x_i} D_{x_j}.
\end{eqnarray*}
Note that as $F$ is convex, the matrix $\big(a^{ij}\big)$ is semi-positive definite. 

Next, we introduce some formulas related to the geometry of the boundary $\partial\Omega$. We mainly follow the
formulas in \cite{AW2}. Let $e_j$ ($1\leq j \leq n$) be the coordinate vectors of $\mathbb{R}^n$.
Then for any vector fields $V = V^ie_i$ and $W = W^ie_i$ defined on $\mathbb{R}^n$, we have
\begin{eqnarray*}
D_VW & = & V^j D_j W^i e_i,\\
    D_V D_W u &=& V^i W^j D^2_{x_i x_j} u + \langle D_V W,Du \rangle.
\end{eqnarray*}
 We further introduce the following notations.
\begin{eqnarray*}
    \langle V,W \rangle_A &=& a^{ij} V^i W^j;\quad \langle V,W \rangle_{\delta} = \delta^{ij} V^i W^j = V^iW^i,
\end{eqnarray*}
and for matrices $P = [p_{ij}]_{n\times n}, \ Q = [q_{ij}]_{n\times n}$, 
\begin{eqnarray*}
    \langle P,Q \rangle_{A,\delta} &=& a^{ij} \delta^{kl} p_{ik} q_{jl},
\end{eqnarray*}
where $\delta^{kl}$ denotes the delta function, $\delta^{kl} = 1$ if $k=l$ and $0$ otherwise.

In order to prove our gradient estimate, we need to differentiate along the boundary $\partial\Omega$. 
To do this, we will extend using parallel transport of the unit tangent $T$ and (inward) normal $N$ vectors of $\partial\Omega$ to a neighborhood of $\partial\Omega$ (along the normal direction). This leads to the following formula.
\begin{lemma}[\cite{AW2}]\label{lem:bd_grad_cal} Let $k$ be the curvature of $\partial\Omega$. Then on $\partial \Omega$, we have:
\begin{enumerate}
\item For $f\in C^{\infty}(\bar{\Omega})$, the mixed derivatives of $f$ is related by
\begin{equation*}
D_N D_T f=D_T D_N f+k D_T f.
\end{equation*}

\item $D_T T=k N;\ D_T N=-k T;\ D_N T=D_N N=0 $.
\item Using the boundary condition in \eqref{eq:ang_pro}, we have 
$$ |D_N u|^2 = \cot^2 \theta ( 1+|D_T u|^2), $$
$$ |D_T u|^2 = \sin^2 \theta ( 1+|Du|^2 ) - 1. $$
\item\label{bdry2diff} For the second derivations, we have
\begin{eqnarray*}
D_N u_t &=&-\cos{\theta} \frac{Du \cdot Du_t}{\sqrt{1+|Du|^2}},\\
D_T D_N u  &=& \sin{\theta}(D_T \theta)v-\cos{\theta}(D_T v),\\
D_N D_T u &=& \sin{\theta}(D_T \theta)v-\cos{\theta}(D_T v)+k D_T u,\\
D_T D_T u&=&\frac{ \sin{\theta} \cos{\theta} (D_T \theta)v^2+\sin ^2\theta \ vD_T v }{D_T u}.
\end{eqnarray*}
\end{enumerate}
\end{lemma}
From item \ref{bdry2diff} above,
all the second order derivatives of $u$ along the boundary $\partial\Omega$ are given except $D_N D_N u$, 
which can only be found using the governing equation \eqref{eq:ansio_mob_mcf}. This is a key step in deriving our gradient estimate.

\section{Main Results with Contact Angle Boundary in $\mathbb{R}^2$ }
\label{sec:ang_grad_est}
In this section, we prove the crucial gradient estimate for \eqref{eq:ang_pro}.
This will imply the long time existence of solution. 

First we state the main assumptions for our results.

\begin{enumerate}
\item[A1.] The domain $\Omega\subset\mathbb{R}^2$ is assumed to be bounded, of class
$C^3$, and uniformly convex in the sense that there are positive constants $k_0, k_1$
such that $k_0 \leq k(x) \leq k_1$ for any $x\in\partial\Omega$.

\item[A2.] For the boundary contact angle $\theta$, there are positive constant 
$\delta_0, \theta_0$ such that 
\[
|D_T\theta(x)|\leq k(x) -\delta_0\quad\text{for any $x\in\partial\Omega$, and}\quad
\theta_0 \leq \theta \leq \pi- \theta_0.
\]
\item[A3.] For the anisotropic function $F$, the constants $m_0$ and $m_2$
from \eqref{eq:cond_aniso} satisfies $m_2 < m_0$.

\item[A4.] The initial condition $u_0$ is assumed to be $C^3(\overline{\Omega})$ and satisfies $D_N u_0 = - \sqrt{1+|Du_0|^2} \cos \theta$ on $\partial \Omega$.
\end{enumerate}
We will assume A1--A4 for the rest of this Section.

\begin{theorem}\label{thm:ang_grad_est}
There exists a positive constant $C_1$, independent of time such that for any solution $u$ to \eqref{eq:ang_pro}, we have,
\begin{equation}\label{eq:ang_grad_est}
    \sup_{Q_T} |Du| \leq C_1.
\end{equation}
\end{theorem}
Before proving the above theorem, similar to \cite{AW2},
we first establish an a priori bound on the time derivative of $u$.
\begin{lemma}\label{lem:ang_ut_est}
The following holds for $u_t$ 
\begin{equation*}
    \sup_{Q_T} |u_t| = \sup_{\Omega_0}|u_t| =: C_2.
\end{equation*}
\end{lemma} 
\begin{proof}
From \eqref{eq:ang_pro} we compute 
\begin{equation} \label{eq:ut_deri}
\begin{split}
    \frac{\partial}{\partial t} |u_t|^2 & = \sum_{i,j} (|u_t|^2)_{x_i x_j} - 2\sum_{i,j} a^{ij} u_{t,x_i} u_{t,x_j} + \sum_{i,j,k} u_{x_i x_j} \frac{\partial a^{ij}}{\partial p^k} (|u_t|^2)_{x_k} \\
    & = \triangle_A |u_t|^2 - 2 \langle Du_t,Du_t  \rangle_A + u_{x_i x_j} \langle D_p a^{ij},D|u_t|^2 \rangle_{\delta}.
\end{split}
\end{equation}
Hence, for the elliptic operator $L(\varphi) := - \triangle_A \varphi - u_{x_i x_j} \langle D_p a^{ij},D\varphi \rangle_{\delta}$, 
we have 
\begin{equation}\label{eq:ut_weak_max}
\frac{\partial}{\partial t} |u_t|^2 + L |u_t|^2 = -2 \langle Du_t,Du_t  \rangle_A \leq 0.
\end{equation}
Thus the weak maximum principle leads to 
$\sup_{Q_T} |u_t|^2=\sup_{\Gamma_T \cup \Omega_0} |u_t|^2 $.

Suppose the maximum of $|u_t|$ occurs at $(x_0,\tau)\in\Gamma_T$. 
Then $D_T u_t(x_0,\tau)=0$. Applying the boundary condition in \eqref{eq:ang_pro} at this point, we have
\begin{equation*}
    D_N u_t = -\cos{\theta} \ \frac{D_N u \ D_N u_t + D_T u \ D_T u_t}{v} = -\cos{\theta} \ \frac{D_N u \ D_N u_t }{v} = \cos^2{\theta} D_N u_t
\end{equation*}
so that $D_N u_t=0$. Now, if $L (|u_t|^2)\leq 0$, then Hopf Lemma gives $D_N u_t<0$ 
which is clearly a contradiction. Hence $L(|u_t|^2)$ must be strictly
positive. Then \eqref{eq:ut_weak_max} implies that
$\frac{\partial}{\partial t} |u_t|^2 < 0$ so that $(x_0, \tau)$ cannot be a
maximum point of $u_t$. Thus we conclude that
$\sup_{Q_T} |u_t| = \sup_{\Omega_0}|u_t|$.
\end{proof}   

Next, we prove that the maximum of $|Du|$ occurs on the parabolic boundary, $\Gamma_T \cup \Omega_0$. 
\begin{lemma} \label{lem:ang_Du_bd}
$\sup_{Q_T} |Du|^2=\sup_{\Gamma_T \cup \Omega_0} |Du|^2 $.
\end{lemma}  
\begin{proof}
Similar to the calculation of above Lemma, we have from \eqref{eq:ang_pro} that,
\begin{equation}
\begin{split}
\frac{\partial}{\partial t} |Du|^2 
& = \sum_{i,j} a^{ij} (|Du|^2)_{x_i x_j} - 2 \sum_{i,j} a^{ij} \sum_k u_{x_k x_i} u_{x_k x_j} + \sum_{i,j,l} u_{x_i x_j} \frac{\partial a^{ij}}{\partial p^l} (|Du|^2)_{x_l} \\
& = \triangle_A |Du|^2 - 2 \langle D^2 u,D^2 u  \rangle_{A,\delta} + u_{x_i x_j} \langle D_p a^{ij},D|Du|^2 \rangle_{\delta}.
\end{split}
\end{equation}
Hence, we have
\[
\frac{\partial}{\partial t} |Du|^2 + L|Du|^2 = -2 \langle D^2 u,D^2 u  \rangle_{A,\delta} \leq 0.
\]
By the weak maximum principle again, we obtain $ \sup_{Q_T} |Du|^2=\sup_{\Gamma_T \cup \Omega_0} |Du|^2 $. 
\end{proof} 

Now we provide the proof of the gradient estimate \eqref{eq:ang_grad_est}. 
From the previous Lemma, we
only need to estimate the gradient on the boundary $\partial\Omega$. To do this, 
we decompose the equation in \eqref{eq:ang_pro} along the tangential and normal directions. Using the formula from Lemma \ref{lem:bd_grad_cal}, the second normal derivative $D_N D_N u$ can be computed. \\ \\
\textbf{Proof of Theorem} \ref{thm:ang_grad_est}.
Note that we just need to consider the case that the maximum of $|Du|^2$ occurs at $(x_0,\tau)\in \Gamma_T$. Suppose $|D_T u| (x_0,\tau)< 1$, by Lemma \ref{lem:bd_grad_cal}, item 3, it holds that
\begin{equation}
    |Du|^2(x_0,\tau)=\left.\frac{|D_T u|^2+1}{\sin^2 \theta}\right|_{x_0}-1 < \frac{2}{\sin^2 \theta_0}-1,
\end{equation}
where in the last step, we have used assumption A2 about $\theta$.
Hence the gradient bound holds in this case. Therefore without loss of generality, 
we can assume that $|D_T u| (x_0,\tau)\geq 1$.

At any maximum point $(x_0,\tau)$ of $|Du|^2$, we have
\begin{equation} \label{eq:ang_max_normal}
    D_N |Du|^2(x_0,\tau) \leq 0,
\end{equation}
\begin{equation}
    D_T |Du|^2(x_0,\tau) = 0.
\end{equation}
The latter in the above implies $D_T v(x_0,\tau)=0$. From Lemma \ref{lem:bd_grad_cal}, we have
\begin{eqnarray} \label{eq:3.8}
D_T D_N u&=&\sin{\theta}(D_T \theta)v,\\
D_N D_T u&=&\sin{\theta}(D_T \theta)v+k D_T u,
\label{eq:3.9}\\
D_T D_T u&=&\frac{ \sin{\theta} \cos{\theta} (D_T \theta)v^2 }{D_T u}.
\label{eq:3.10}
\end{eqnarray}
We can then rewrite \eqref{eq:ang_max_normal} as
\begin{equation} \label{eq:ang_ineq1}
    D_N u \ D_N D_N u (x_0,\tau) + D_T u \ D_N D_T u (x_0,\tau) \leq 0.
\end{equation}

Consider the orthonormal matrix $O=(T,N)$. In order to obtain $D_N D_N u$, 
we will rewrite \eqref{eq:ang_pro} by decomposing the elliptic operator into its tangential and normal components:
\begin{equation} \label{eq:ang_ut_rew}
\begin{split}
    u_t & = \text{tr} \left( A\,D^2u \right) \\
        & = \text{tr} \left( O^t A O O^t \ D^2 u \ O  \right) \\
        & =  \left( T^t A T \right) \left( T^t \ D^2 u \  T \right) 
        +  \left( T^t A N \right) \left( N^t \ D^2 u \  T \right) \\
        & \ \ \ \  +  \left( N^t A T \right) \left( T^t \ D^2 u \ N \right) + \left( N^t A N \right) \left( N^t \ D^2 u \ N \right) \\
        & =  a^{T T} \left( D_{T} D_{T} u -\langle D_{T} T , Du \rangle   \right) +  a^{T N} \left( D_{T} D_{N} u -\langle D_{T} N , Du \rangle   \right) \\
        & \ \ \ \  +  a^{N T} \left( D_{N} D_{T} u -\langle D_{N} T , Du \rangle   \right) + a^{NN} \left( D_{N} D_{N} u -\langle D_{N} N , Du \rangle   \right) \\
        & =  a^{T T} D_{T} D_{T} u  - k a^{T T} D_N u +  a^{T N} D_{T} D_N u +  k a^{T N} D_{T} u \\
        & \ \ \ \  +  a^{N T} D_N D_{T} u + a^{NN} D_N D_N u,
\end{split}
\end{equation}
where we have used again the formulas from Lemma \ref{lem:bd_grad_cal}
combined with the following notations:
\begin{equation}\label{eq:att.2d}
a^{T T}  := T^t A T ,\quad
a^{T N}  := T^t A N ,\quad
a^{N T}  := N^t A T ,\quad\text{and}\quad
a^{NN}   := N^t A N.
\end{equation}
(Note that the above are all scalars, so that $a^{TN} = a^{NT}$.)
Hence we can solve for $D_N D_N u$ by expressing it in terms of other second order derivatives of $u$ and $u_t$,
\begin{equation} \label{eq:ang_NNu}
    a^{NN} D_N D_N u = u_t - a^{TT} D_T D_T u - a^{TN} D_T D_N u - a^{NT} D_N D_T u + k a^{TT} D_N u - k a^{TN} D_T u.
\end{equation}

Next, we calculate $a^{TT},\ a^{TN},\ a^{NN}\ $ at $(x_0,\tau)$. 
For simplicity, 
we write
\begin{eqnarray*}
    F &=& F\left(\frac{Du,-1}{v}\right),\\
    G &=& G\left(\frac{Du,-1}{v}\right),\\
    D_{p_i}(F) &=& D_{p_i} F(\xi) \Bigg|_{p=(Du,-1)},\quad\text{for $i=1,2$,}\\
    D^2_{p_i p_j} (F) &=& D^2_{p_i p_j} F(\xi) \Bigg|_{p=(Du,-1)},\quad\text{for $i,j=1,2$,}
\end{eqnarray*}and 
\begin{equation*}
    D_p (F) = (D_{p_1}(F),D_{p_2}(F)),
\end{equation*}
\begin{equation*}
    D^T_p (F) = \langle D_p (F), T \rangle,\ \ D^N_p (F) = \langle D_p (F), N \rangle,
\end{equation*}
\begin{equation*}
    F^{TT}=T^t [D^2(F)]  T,   
\end{equation*}
\begin{equation*}
    F^{TN}=F^{NT}=N^t [D^2(F)]  T,
\end{equation*}
\begin{equation*}
    F^{NN}=N^t [D^2(F)]  N.
\end{equation*}
(The above expressions are all evaluated at $(x_0,\tau)$.)
From Section \ref{sec:prelim}, we have
\begin{equation} \label{eq:ang_F_sep}
\begin{split}
    F(Du,-1) &= v F, \\
    D_{p_i} F (Du,-1) &= \frac{u_{x_i}}{v} F + v D_{p_i} (F),\\
    D^2_{p_i p_j} F (Du,-1) &= \frac{1}{v} \left( \delta_{ij} - \frac{u_{x_i} u_{x_j}}{v^2} \right) F + \frac{u_{x_i}}{v} D_{p_j} (F) + \frac{u_{x_j}}{v} D_{p_i} (F) + v D^2_{p_i p_j} (F),
\end{split}
\end{equation}
so that we can rewrite $a^{ij}$ as
\begin{equation}
    a^{ij}= \left( \delta_{ij} - \frac{u_{x_i} u_{x_j}}{v^2} \right) F G + u_{x_i} D_{p_j} (F) \ G + u_{x_j} D_{p_i} (F) \ G + v^2 G D^2_{p_i p_j} (F).
\end{equation}
Hence from \eqref{eq:att.2d}, we get
\begin{equation} \label{eq:ang_aTN_sep}
\begin{split}
    & a^{TT}=\frac{1+|D_N u|^2}{v^2} F G + 2 D_T u\ G\ D^T_p (F) + v^2 G F^{TT} \\
    &\ \ \ \ \ =I_1(a^{TT}) + I_2(a^{TT}) + I_3(a^{TT}), \\
    & a^{TN}=-\frac{D_T u \ D_N u}{v^2} F G + \left( D_T u\ D^N_p (F) + D_N u\ D^T_p (F) \right) G + v^2 G F^{TN} \\
    &\ \ \ \ \ =I_1(a^{TN}) + I_2(a^{TN}) + I_3(a^{TN}), \\
    & a^{NN}=\frac{1+|D_T u|^2}{v^2} F G + 2 D_N u\ G\ D^N_p (F) + v^2 G F^{NN} \\
    &\ \ \ \ \ =I_1(a^{NN}) + I_2(a^{NN}) + I_3(a^{NN}),
\end{split}    
\end{equation}
where in the above, we have decomposed each expression into three components $I_1,I_2$,
and $I_3$ in such a way that
$I_1$ consists of $F$, 
$I_2$ consists of $D_{p_i}(F)$ and 
$I_3$ consists of $D^2_{p_i p_j} (F)$. 

Upon substituting the expression \eqref{eq:ang_NNu} of $D_N D_N u$ into equation \eqref{eq:ang_ineq1} and multiplying both sides by $a^{NN}$, we have
\begin{equation}
\begin{split} \label{eq:ang_ineq2}
    0 \geq & D_N u \Big( u_t - a^{TT} D_T D_T u - a^{TN} D_T D_N u - a^{NT} D_N D_T u + k a^{TT} D_N u - k a^{TN} D_T u \Big) \\ 
    & + a^{NN} D_T u \  D_N D_T u.
\end{split}
\end{equation}
By the boundary condition in \eqref{eq:ang_pro} $D_N u= -v\cos{\alpha}$, and using
the identities \eqref{eq:3.8}--\eqref{eq:3.10}, and \eqref{eq:ang_F_sep},  
the above becomes 
\begin{equation} \label{eq:ang_ineq3}
\begin{split}
    u_t v \cos{\theta} & \geq  \Big(a^{TT} D_T D_T u + a^{TN} D_T D_N u + a^{TN} D_N D_T u - k a^{TT} D_N u + k a^{TN} D_T u\Big) v \cos{\theta} \\
    & \ \ \ + a^{NN} D_T u \ D_N D_T u \\
    & =  a^{TT} \frac{ \sin{\theta} \cos^2\theta (D_T \theta) v^3 }{D_T u} 
    + a^{TN} \sin{\theta} \cos{\theta} (D_T \theta) v^2 \\
    & \ \ \ + a^{TN} \left(\sin{\theta} \cos{\theta} (D_T \theta) v^2 + k \cos{\theta} (D_T u) v\right) \\
    & \ \ \ + a^{TT} k \cos^2\theta \ v^2
    + a^{TN} k \cos{\theta} (D_T u) v
    + a^{NN} \left(\sin{\theta} (D_T \theta) v + k D_T u\right) D_T u \\
    & = \left(I_1(a^{TT}) + I_2(a^{TT}) + I_3(a^{TT})\right) \left(\frac{ \sin{\theta} \cos^2\theta (D_T \theta) v^3 }{D_T u} + k \cos^2\theta \ v^2\right) \\
    & \ \ \ + \left(I_1(a^{TN}) + I_2(a^{TN}) + I_3(a^{TN}) \right) \left(2\sin{\theta} \cos{\theta} (D_T \theta) v^2 + 2k \cos{\theta} (D_T u) v \right) \\
    & \ \ \ + \left(I_1(a^{NN}) + I_2(a^{NN}) + I_3(a^{NN}) \right) \left(\sin{\theta} (D_T \theta) v + k D_T u \right) D_T u \\
    & = J_1 +J_2 +J_3,
\end{split}
\end{equation} 
where for $i=1,2,3$, $J_i$ is the sum of the expressions containing 
$I_i(a^{TT})$, $I_i(a^{TN})$ and $I_i(a^{NN})$ in \eqref{eq:ang_ineq3}. In the following,
we compute each of these terms in detail. 
\begin{equation} \label{eq:ang_J1}
    \begin{split} 
        J_1 = & I_1(a^{TT})\left(\frac{ \sin{\theta} \cos^2\theta (D_T \theta) v^3 }{D_T u} + k \cos^2\theta \ v^2\right) \\
        & + I_1(a^{TN}) \left(2\sin{\theta} \cos{\theta} (D_T \theta) v^2 + 2k \cos{\theta} (D_T u) v \right) \\
        & + I_1(a^{NN}) \left(\sin{\theta} (D_T \theta) v + k D_T u \right) D_T u  \\
        = & \Bigg( \frac{1+|D_N u|^2}{v^2} \left(\frac{ \sin{\theta} \cos^2\theta (D_T \theta) v^3 }{D_T u} + k \cos^2\theta \ v^2 \right)  \\
        & - \frac{D_T u \ D_N u}{v^2} \left(2\sin{\theta} \cos{\theta} (D_T \theta) v^2 + 2k \cos{\theta} (D_T u) v \right) \\
        &  + \frac{1+|D_T u|^2}{v^2} \left(\sin{\theta} (D_T \theta) v + k D_T u\right) D_T u  \Bigg)  F G \\
        = & \Bigg( (1+v^2 \cos^2 \theta) \left(\frac{ \sin{\theta} \cos^2\theta (D_T \theta) v }{D_T u} + k \cos^2\theta \right) \\
        & + 2 \sin{\theta} \cos^2 \theta (D_T \theta) v D_T u + 2k \cos^2 \theta |D_T u|^2 \\ 
        &  +  \sin^2 \theta \left(\sin{\theta} (D_T \theta) v D_T u + k |D_T u|^2 \right)  \Bigg) F G \\
        = & \Bigg( \sin{\theta} \cos^2 \theta (D_T \theta) \frac{v}{D_T u} + \sin{\theta} \cos^4 \theta (D_T \theta) \frac{v^3}{D_T u} +k \cos^2 \theta (1+ v^2 \cos^2 \theta) \\
        & + 2 \sin{\theta} \cos^2 \theta (D_T \theta) v D_T u  + 2 k \cos^2 \theta |D_T u|^2 \\
        & + \sin^3 \theta (D_T \theta) v D_T u + k \sin^2 \theta (v^2 \sin^2 \theta - 1)  \Bigg ) F G \\
        = & \left(  \left( \frac{v^3}{D_T u} \cos^2 \theta + v D_T u \right) \sin{\theta} D_T \theta  + k(v^2-1)      \right) F G  \\
        = & \left( \sin{\theta} D_T \theta  \frac{v(v^2-1)}{D_T u} + k(v^2-1)  \right) 
        F G.
    \end{split}
\end{equation}
Note that the above expression is well-defined as $|D_Tu| > 1$.
\begin{equation} \label{eq:ang_J2}
    \begin{split}
        J_2 = & I_2(a^{TT}) \left(\frac{ \sin{\theta} \cos^2\theta (D_T \theta) v^3 }{D_T u} + k \cos^2\theta \  v^2 \right) \\ 
        & + I_2(a^{TN}) \left(2\sin{\theta} \cos{\theta} (D_T \theta) v^2 + 2k \cos{\theta} (D_T u) v \right) \\
        & + I_2(a^{NN}) \left(\sin{\theta} (D_T \theta) v + k D_T u \right) D_T u \\
        = & \Bigg( 2 D_T u\ D^T_p (F)  \left(\frac{ \sin{\theta} \cos^2\theta (D_T \theta) v^3 }{D_T u} + k \cos^2\theta \ v^2 \right)  \\  
        & + 2 \left(D_T u\ D^N_p (F) + D_N u\ D^T_p (F) \right) \left( \sin{\theta} \cos{\theta} (D_T \theta) v^2 + k \cos{\theta} (D_T u) v \right)  \\
        & + 2 D_N u\ D^N_p (F) \left(\sin{\theta} (D_T \theta) v + k D_T u \right) D_T u    \Bigg) G \\
        = & \Bigg(  D^T_p (F) \Big(\sin{\theta} \cos^2\theta (D_T \theta) v^3 + k \cos^2\theta (D_T u) v^2 \\ 
        &+ \sin{\theta} \cos{\theta} (D_T \theta) (D_N u) v^2 + k \cos{\theta} (D_T u) (D_N u) v  \Big) \\
        & + D^N_p (F) \Big( \sin{\theta} \cos{\theta} (D_T \theta) (D_T u) v^2 + k \cos{\theta} |D_T u|^2 v \\  
        & + \sin{\theta} (D_T \theta) (D_T u) (D_N u) v +  k |D_T u|^2 D_N u  \Big) \Bigg) 2G \\
        = & 0,
    \end{split}
\end{equation}
where the vanishing of $J_2$ is due to the boundary condition
$D_N u= -v\cos{\theta}$.

\begin{equation}  \label{eq:ang_J3}
    \begin{split}
        J_3 = & I_3(a^{TT}) \left(\frac{ \sin{\theta} \cos^2\theta (D_T \theta) v^3 }{D_T u} + k \cos^2\theta \ v^2 \right) \\ 
        & + I_3(a^{TN}) \left(2\sin{\theta} \cos{\theta} (D_T \theta) v^2 + 2k \cos{\theta} (D_T u) v \right) \\
        & + I_3(a^{NN}) \left(\sin{\theta} (D_T \theta) v + k D_T u \right) D_T u  \\
        = & \Bigg(  F^{TT} \left(\frac{ \sin{\theta} \cos^2\theta (D_T \theta) v^3 }{D_T u} + k \cos^2\theta \ v^2 \right) + 2 F^{TN} \left(\sin{\theta} \cos{\theta} (D_T \theta) v^2 + k \cos{\theta} (D_T u) v\right)  \\
        & + F^{NN} \left(\sin{\theta} (D_T \theta v) + k D_T u \right) D_T u   \Bigg) v^2 G \\
        = & \Bigg( \sin{\theta} D_T \theta  \left( F^{TT} \cos^2 \theta \frac{v^3}{D_T u} + 2 F^{TN} \cos{\theta} \ v^2 + F^{NN} v D_T u \right) \\
        & + k \left( F^{TT} \cos^2 \theta \ v^2 + 2 F^{TN} \cos{\theta} \ v D_T u + F^{NN} |D_T u|^2 \right)   \Bigg) v^2 G \\
        = & \left( \sin{\theta} D_T \theta \frac{v}{D_T u} + k \right) \Big\langle (v \cos{\theta}) T + (D_T u) N , (v \cos{\theta}) T + (D_T u)  N \Big\rangle_{D^2 (F)} v^2 G,
    \end{split}
\end{equation}
where in the above, we have used the notation $\langle V,W \rangle _{D^2 (F)}:=D^2_{p_i p_j} (F) \ V^i W^j$.

To continue, by Lemma \ref{lem:bd_grad_cal}, item 3, we have $ \|(v \cos{\theta}) T + (D_T u) N\|=\sqrt{v^2-1}$. From \eqref{eq:cond_aniso}, we infer 
\begin{equation}
\begin{split}
    \Big\langle (v \cos{\theta}) T + (D_T u) N , (v \cos{\theta}) T + (D_T u) N \Big\rangle_{D^2 (F)} 
    & \leq \|D^2(F)\| \Big\|{(v \cos{\theta}) T + (D_T u) N} \Big\|^2 \\
    & \leq \frac{m_2 (v^2-1)}{v^2}.
\end{split}
\end{equation}
Upon adding \eqref{eq:ang_J1}--\eqref{eq:ang_J3}, inequality \eqref{eq:ang_ineq2} then becomes
\begin{equation} \label{eq:ang_ineq4}
\begin{split}
    |u_t| v |\cos{\theta}| \geq & u_t v \cos{\theta} \\
    = & \left( \sin{\theta} \ D_T \theta \frac{v}{D_T u} + k \right) \Bigg((v^2-1)F \\ & + \Big\langle (v \cos{\theta}) T + (D_T u) N , (v \cos{\theta}) T + (D_T u) N \Big\rangle_{D^2 (F)} v^2 \Bigg) G \\
    \geq & \left( \sin{\theta} \ D_T \theta \frac{v}{D_T u} + k \right) \left( F - m_2 \right) (v^2-1) G.
\end{split}
\end{equation}
By Lemma \ref{lem:bd_grad_cal}, item 3, the following holds: 
\begin{equation}
    v^2-1 = \frac{|D_T u|^2}{\sin^2 \theta} + \frac{\cos^2 \theta}{\sin^2 \theta}.
\end{equation}
Then \eqref{eq:ang_ineq4} becomes
\begin{equation}
\begin{split}
    \frac{|u_t| v |\cos{\theta}|}{G \left(F - m_2 \right)} & \geq \left( \sin{\theta} \ D_T \theta   \frac{v}{D_T u} + k \right) \left( \frac{|D_T u|^2}{\sin^2 \theta} + \frac{\cos^2 \theta}{\sin^2 \theta}  \right) \\
    & = D_T \theta \frac{D_T u}{\sin{\theta}} v + D_T \theta \frac{\cos^2 \theta}{D_T u\, \sin{\theta}} v + k(v^2-1).
\end{split}
\label{eq:ang_ineq5}
\end{equation}
Since $|\frac{D_T u}{\sin{\theta}}|\leq v$, $v\geq 1$,  and by assumption $|D_T u|\geq 1$, upon dividing both sides of \eqref{eq:ang_ineq5} by $v$, we obtain 
\begin{equation}
\begin{split}
    (k-D_T \theta)v & \leq \frac{|u_t| |\cos{\theta}|}{G \left(F - m_2 \right)} + \frac{k}{v} - D_T \theta \frac{\cos^2 \theta}{D_T u\, \sin{\theta}} 
    \leq \frac{|u_t| |\cos{\theta}|}{G \left(F - m_2 \right)} + k + \frac{k}{\sin{\theta}}.
\end{split}
\end{equation}
Then by Lemma \ref{lem:ang_ut_est}, $|u_t|\leq C_2$ and assumptions A2 and A3, we 
can conclude that 
\begin{equation}
    v \leq \frac{1}{\delta_0} \left( \frac{C_2}{g_0 (m_0-m_2)} + k_0 + \frac{k_0}{\sin{\theta_0}} \right),
\end{equation}
finishing the proof of the gradient estimate \eqref{eq:ang_grad_est}.
\hfill$\square$

Next, we will construct a translating solution to \eqref{eq:ang_pro} in the form 
$u(x,t) = w(x) + \lambda t$ for some function $w$ and constant $\lambda$.
Then $w$ and $\lambda$ satisfy the following elliptic problem:
\begin{equation*} 
\left\{ \begin{array}{cllcl}
    \lambda & =& G (Dw,-1) \ D^2_{p_i p_j} F (Dw,-1) w_{x_i x_j}    \quad & \text{in} & \ \Omega,  \\
     D_N w & = & - \sqrt{1+|Dw|^2} \cos \theta \quad & \text{on} & \ \partial \Omega,
\end{array} \right.
\end{equation*}
where $\lambda$ is to represent the speed of a translating invariant solution.
Similar to \cite{AW2}, the above equation is solved by first considering the following nonlinear eigenvalue problem:
\begin{equation} \label{eq:ang_eign_pro}
\left\{ \begin{array}{cllcl}
     \epsilon w^{\epsilon} &=& G (Dw^{\epsilon},-1) \ D^2_{p_i p_j} F (Dw^{\epsilon},-1) w^{\epsilon}_{x_i x_j}    \quad & \text{in} & \ \Omega,  \\
     D_N w^{\epsilon} &=& - \sqrt{1+|Dw^{\epsilon}|^2} \cos \theta \quad & \text{on} & \ \partial \Omega. 
\end{array} \right.
\end{equation}
In fact, the technique in the proof of Theorem 3.1 is also applicable to the above equation. 
For what follows, we basically use the approach in \cite{AW2} and \cite{GMWW}.
\begin{theorem}
Under the assumptions in Theorem 3.1, there is a unique $\lambda$ such that a solution exists for \eqref{eq:ell_ang_pro}.
In this case, the solution is also unique.
\end{theorem}
\begin{proof}
Firstly, we want to show that there exists a positive constant $C$ independent of $\epsilon$ such that $|\epsilon w^{\epsilon}|\leq C$. 

For this purpose, consider the function $\psi=a_0 d$ defined in a neighborhood of $\partial\Omega$.
Here $d$ is the distance function to $\partial\Omega$ and $a_0$ is a constant such that 
\begin{equation*}
    a_0 < -\cos{\theta(\cdot)} \sqrt{1+a_0^2}.
\end{equation*}
(By assumption A2, such an $a_0$ exists.)
Then we extend $\psi$ smoothly to all of $\Omega$. Note that
$D_N\psi = a_0$ on $\partial\Omega$. Now suppose $\psi-w_{\epsilon}$ attains its minimum at $x_0$. We have following two cases: \\
\textbf{Case 1:} $x_0 \in \partial\Omega$. This implies $D_T \psi(x_0) = D_T w^{\epsilon}(x_0)$ and $D_N \psi(x_0) \geq D_N w^{\epsilon}(x_0)$. Hence at $x_0$, we have
\begin{equation*}
      -\cos{\theta} = \frac{D_N w^{\epsilon}}{\sqrt{1+|D w^{\epsilon}}|^2} \leq \frac{D_N \psi}{\sqrt{1+|D \psi|^2}} < -\cos{\theta},
\end{equation*}
which is a contradiction. (In the above, we have used the facts that 
$|Dw^\epsilon|^2 = |D_Tw^\epsilon|^2 + |D_Nw^\epsilon|^2
\leq
|D_T\psi|^2 + |D_N\psi|^2
=|D\psi|^2
$ and the function $f(x) = \frac{x}{\sqrt{c^2+x^2}}$ is increasing 
in $x\in\mathbb{R}$.)\\
\textbf{Case 2:} $x_0 \in \Omega$. This implies $D\psi(x_0) = D w^{\epsilon}(x_0)$ and $\big[ D^2 \psi(x_0) \big] \geq \big[ D^2 w^{\epsilon}(x_0) \big]$. Hence at $x_0$, for some constant $C=C(\theta, \Omega)$,
\begin{equation*}
    \epsilon w^{\epsilon} = a^{ij} (D w^{\epsilon}) w_{x_i x_j}^{\epsilon} \leq a^{ij} (D \psi) \psi_{x_i x_j} \leq C,
\end{equation*}
which suggests an upper bound for $\epsilon w^{\epsilon}$. Similarly, considering the maximum of $\psi-w_{\epsilon}$ suggests a lower bound. Hence $|\epsilon w^{\epsilon}|$ is bounded.

Once the above is achieved, we can then directly apply the proof of Theorem 3.1 by replacing $u_t$ with $\epsilon w^{\epsilon}$ to obtain an uniform estimate for $|Dw^{\epsilon}|$.
This implies $| D (\epsilon w^{\epsilon})| \rightarrow 0$ as $\epsilon \rightarrow 0$. Hence $\epsilon w^{\epsilon} \rightarrow \lambda$ for some constant $\lambda$.

For the uniqueness, assume $u^1$ and $u^2$ are two solutions of \eqref{eq:ell_ang_pro} with $\lambda_1$ and $\lambda_2$ respectively and $\lambda_1 \leq \lambda_2$. Let $\Tilde{u}=u^1 - u^2$. Then we have
\begin{equation*}
    \Tilde{a}^{ij} \Tilde{u}_{x_i x_j} + \Tilde{b}^{i} \Tilde{u}_{x_i} \leq 0,
\end{equation*}
where $\Tilde{a}^{ij} = a^{ij} (Du^1)$ and $\Tilde{b}^{i} = u^2_{x_k x_l} \int^1_0 a^{kl}_{p_i} (\eta Du^1 + (1-\eta) Du^2 ) d\eta $. The maximum principle suggests that 
$\max \Tilde{u}$ occurs on the boundary, $x_0 \in \partial \Omega$. Then Hopf lemma implies that at $x_0$, we have $D_N u^1 < D_N u^2$. Together with $D_T u^1 = D_T u^2$, this would then contradict the following,  
\begin{equation*}
    \frac{D_N u^1}{\sqrt{1+|D u^1|^2}} =  \frac{D_N u^2}{\sqrt{1+|D u^2|^2}}
    \,\,(=\cos\theta).
\end{equation*}
Hence $\Tilde{u}$ must be a constant and $\lambda$ is unique.
\end{proof} 
From \textbf{Theorem 3.4}, the solution $w(x)$ to the elliptic problem \eqref{eq:ell_ang_pro} gives a translating solution $\Tilde{w}(x,t) = w(x) + \lambda t$ which solves the following parabolic boundary problem:
\begin{equation} 
    \left\{ \begin{array}{cllcl}
            u_t & = & G(Du,-1) \  D^2_{p_i p_j} F (Du,-1) u_{x_i x_j}   \quad &  \text{in} & \ Q_T, \\
            D_N u & = & - \sqrt{1+|Du|^2} \cos \theta  \quad &  \text{on} & \ \Gamma_T,  \\
         u(\cdot,0) & = & w(\cdot)  \quad \ & \text{in} & \   \Omega_0.
    \end{array}  \right.
\end{equation}

Next we show some convergence property of $u$ as $t\longrightarrow\infty$.
\begin{corollary}
For a solution $u=u(x,t)$ of \eqref{eq:ang_pro}, there exists a constant $C$ such that
\begin{equation}
    |u(x,t)-\lambda t| \leq C.
\end{equation}   
\end{corollary}
\begin{proof}
 Consider $\Tilde{u}(x,t) = u(x,t) - \Tilde{w}(x,t)$. It satisfies the following parabolic equation
\begin{equation*}
    \Tilde{u}_t = \Tilde{a}^{ij} \Tilde{u}_{x_i x_j} + \Tilde{b}^{i} \Tilde{u}_{x_i},
\end{equation*}
where $\Tilde{a}^{ij} = a^{ij} (Du)$ and $\Tilde{b}^{i} = (\Tilde{w})_{x_k x_l} \int^1_0 a^{kl}_{p_i} (\eta Du + (1-\eta) D\Tilde{w} ) d\eta $. If $\Tilde{u}$ attains its maximum and minimum at $t=0$, the result follows immediately. Next,
suppose $\Tilde{u}$ attains its maximum and minimum at $(x_0,\tau) \in \partial\Omega \times (0,\infty)$. Similar to the proof in Theorem 3.4, we can also conclude that $\Tilde{u}$ is a constant. Hence the result follows.
\end{proof}
Finally, we show the uniqueness, up to a constant, of long time limit to \eqref{eq:ang_pro}.
\begin{theorem}
Suppose $u^1$ and $u^2$ are two distinct solutions of \eqref{eq:ang_pro} with initial data $u_0^1$ and $u_0^2$. Then $u^1-u^2$
converges to a constant function as $t\rightarrow \infty$. In particular, any solution $u$ to \eqref{eq:ang_pro} 
converges to $\Tilde{w}=\lambda t+w$, up to a constant.  
\end{theorem}
\begin{proof}
We observe that $\Tilde{u}=u^1-u^2$ satisfies the following linear parabolic equation
\begin{equation}  \label{eq:ang_diff_para}
    \left\{ \begin{array}{lcl}
            \Tilde{u}_t = \Tilde{a}^{ij} \Tilde{u}_{x_i x_j} + \Tilde{b}^{i} \Tilde{u}_{x_i}  \quad &  \text{in} & \ \Omega \times (0,+\infty), \\
            \Tilde{c}^{ij} \Tilde{u}_{x_i} N^j = 0 &  \text{on} & \ \partial \Omega \times (0,+\infty),
    \end{array}  \right.
\end{equation}
where $\Tilde{a}^{ij} = a^{ij} (Du^1)$,  $\Tilde{b}^{i} = (u^2)_{x_k x_l} \int^1_0 a^{kl}_{p_i} (\eta Du^1 + (1-\eta) Du^2 ) d\eta $ and $\Tilde{c}^{ij} = \int_0^1 D_{p_j} h^i (\eta Du^1 + (1-\eta) Du^2 ) d\eta$  (define $h^i(p) = \frac{p^i}{\sqrt{1+|p|^2}}$ and hence $[\Tilde{c}^{ij}]$ is a positive definite matrix). The strong maximum principle implies that the oscillation of $\Tilde{u}$
\begin{equation*}
    \text{osc}(\Tilde{u}) (t) = \max_{\Bar{\Omega}} \Tilde{u}(x,t) - \min_{\Bar{\Omega}} \Tilde{u}(x,t),
\end{equation*}
is strictly decreasing in $t$. The result follows if
\begin{equation*}
    \lim_{t\rightarrow \infty} \text{osc}(\Tilde{u}) (t) =0.
\end{equation*}
Suppose otherwise that $\lim_{t\rightarrow \infty} \text{osc}(u) (t) = \delta>0$. Then for some sequence $t_n \rightarrow +\infty$, we define
\begin{equation*}
    u^{1,n} (x,t) := u^1 (x,t+t_n) - \lambda t_n,
\end{equation*}
and 
\begin{equation*}
    u^{2,n} (x,t) := u^2 (x,t+t_n) - \lambda t_n.
\end{equation*}
Notice that by the last Corollary, we have $|u^{i,n}- \lambda t| \leq C$. 
Then there exists a subsequence of $u^{1,n}$ and $u^{2,n}$ converging locally uniformly to $u^{1,\ast}$ and $u^{2,\ast}$ respectively. Let $\Tilde{u}^{\ast} = u^{1,\ast} - u^{2,\ast}$. Then we have
\begin{equation*}
    \text{osc}(\Tilde{u}^{\ast}) (t) = \lim_{n\rightarrow \infty} \text{osc}(\Tilde{u}^{\ast}) (t+t_n) = \delta.
\end{equation*}
Notice that $\Tilde{u}^{\ast}$ satisfies a parabolic equation similar to \eqref{eq:ang_diff_para}. By the strong maximum principle and Hopf lemma, $\Tilde{u}^{\ast}$ is a constant contradicting the fact that $\text{osc}(\Tilde{u}^{\ast}) (t) = \delta$.
\end{proof}

\section{Dirichlet Problem in $\mathbb{R}^n$}
We found that the method in the proof of Theorem \ref{thm:ang_grad_est} can be applied to obtain a gradient bound on the solution to the following Dirichlet boundary value problem in general dimensions.
\begin{equation} \label{eq:diri_pro}
    \left\{ \begin{array}{cllcl}
            u_t  &=& G(Du,-1) \  D^2_{p_i p_j} F (Du,-1) u_{x_i x_j}   \quad &  \text{in} & \ Q_T, \\
            u &=& f(\cdot,t)  \quad &  \text{on} & \ \Gamma_T  \\
         u(\cdot,0) &=& u_0(\cdot)  \quad \ & \text{in} & \   \Omega_0.
    \end{array}  \right.
\end{equation}
In order to utilize some existing results, we find it convenient to assume that
$f$ is of class $C^3$. The precise condition on $f$ will be stated just before the main theorem below.

Before proceeding, we need to generalize the geometric formulas from Section 2 to $\mathbb{R}^n$. 
Let $x\in \partial\Omega$. Choose an orthonormal basis $T_i$, $1\leq i\leq n-1$, such that each $T_i$ is along a principal direction of $\partial\Omega$ with principal 
curvature $k_i$. As before, let also $N$ be the inward normal to $\partial \Omega$.
Upon extending $T_i$ and $N$ to a neighborhood of $\partial\Omega$ by means of parallel transport along the normal direction of $\partial\Omega$, we have the following formula.
\begin{enumerate}
    \item $D_{T_i} T_i= k_i N $ and $D_{T_j} T_i= 0 $ for $i\neq j$. 
    
    \item $D_{T_i} N = -k_i T_i$.  

\item 
For $D_N T_i$, by considering the geodesic in the direction $T_i$, we have a similar formula as in two dimensions,
\begin{equation} 
    D_N D_{T_i} f = k_i D_{T_i} f + D_{T_i} D_N f.
 \end{equation}
 which implies $D_N T_i = 0$.
 
\item The fact $D_N N = 0$ shows by differentiating $|N|^2=1$ and using $N\cdot T_i =0 $. 
\end{enumerate}

Now consider the boundary value problem \eqref{eq:diri_pro}. We assume there is a constant $M$ such that the boundary
function $f$ is $C^3$ satisfying the following,
\begin{equation}\label{eq:diri_f_ass}
    |f_t|,\ |D_{T_i} f|,\ |D_{T_i T_j} f| \leq M.
\end{equation}
At each point on the boundary $\Gamma_T$, we order the principal curvatures $k_i$ such that for some $\alpha$, $1 \leq \alpha \leq n-1$, we have 
\begin{equation}
    k_1 \geq \cdots \geq k_{\alpha} \geq 0 > k_{\alpha+1} \geq \cdots k_{n-1}.
\end{equation}
Let $\max_i |k_i| <
\infty$. Then we state the corresponding gradient estimate for the solution $u$ of 
\eqref{eq:diri_pro} in the following theorem.
\begin{theorem} \label{thm:ell_grad_est}
Let $M_0,\ m_0,\ m_1,\ m_2,\ g_0,\ G_0$ be given in \eqref{eq:cond_mob} and \eqref{eq:cond_aniso}. Consider the following two constants
\begin{equation}
\begin{split}
    \gamma_1 & := g_0 \left( \frac{2}{2+M^2} m_0 -\sqrt{2} M m_1 - m_2 \right), \\  
    \gamma_2 & := G_0 \left( M_0 +\sqrt{2} M m_1 +m_2 \right).
\end{split}
\end{equation}
Assume that they satisfy the following conditions:
\begin{equation} \label{eq:diri_bd_cur_cond}
    \gamma_1 > 0, \quad \gamma_1 \sum_{i=1}^{\alpha} k_i + \gamma_2\sum_{i=\alpha+1}^{n-1} k_i >0.
\end{equation}
Then there exists a constant $C=C(u_0, M, \partial\Omega, F, n )$ such that
\begin{equation*}
    \sup_{Q_T} |Du| \leq C.
\end{equation*}
\end{theorem}
Here we make some remarks about the assumptions and statement.
\begin{enumerate}
\item 
The condition $\gamma_1 > 0$ holds for small enough $m_1, m_2$ together
with small $M$. The former means that  $F$ is close to 
the isotropic function $\overline{F}$ while the latter indicates that
the Dirichlet boundary condition $f$ cannot change too quickly along the boundary.

\item Since $\gamma_1 < \gamma_2$, the second condition in \eqref{eq:diri_bd_cur_cond} holds when 
$\partial\Omega$ is convex, i.e., $\alpha=n-1$ or if $\left|\sum_{i=\alpha+1}^{n-1} k_i\right|$ is small.

\item Similarly to the strategy in the proof of Theorem \ref{thm:ang_grad_est}, we apply maximum principle to obtain a gradient bound on the boundary. In the current Dirichlet case, all we need is an upper bound for $D_N u$ as all the tangential derivatives of $u$ are given 
by the boundary function $f$.
\end{enumerate}

\begin{proof} 
Notice that Lemmas \ref{lem:ang_ut_est} and \ref{lem:ang_Du_bd} also hold in general dimensions, i.e. for some positive constant $C_3=C_3(u_0, M)$, we have
\begin{equation}
\sup_{Q_T} |u_t| \leq C_3,\,\,\,\text{and}\,\,\,
\sup_{Q_T} |Du|^2=\sup_{\Gamma_T \cup \Omega_0} |Du|^2.
\end{equation}

Now suppose the maximum of $|Du|$ occurs at $(x_0, \tau)\in \Gamma_T$. Note that on $\partial \Omega$,
\begin{equation*}
    |Du|^2=\sum_k |D_{T_k} u|^2 + |D_N u|^2
\,\,\,\text{and}\ \ \ 
D_{T_k} u = D_{T_k} f.
\end{equation*}
Hence, we only need to estimate $|D_N u|$ at $(x_0, \tau)$. Assume $|D_N u|\geq 1$. At this point, by maximum principle, we have
\begin{eqnarray} \label{eq:ell_max_norm}
    D_N |Du|^2 &\leq& 0,\\
\label{eq:ell_max_tan}
    D_{T_i} |Du|^2 &=& 0,\ 1\leq i \leq n-1.
\end{eqnarray}
The last equation becomes
\begin{equation*}
    D_{T_i} \left( |D_N u|^2 + \sum_k |D_{T_k} u|^2 \right)=0,
\end{equation*}
\begin{equation*}
    D_N u \ D_{T_i} D_N u + \sum_k D_{T_k} u \ D_{T_i} D_{T_k} u = 0,
\end{equation*}
\begin{equation}
    D_{T_i} D_N u = -\frac{\sum_k D_{T_k} u \ D_{T_i} D_{T_k} u}{D_N u}.
\end{equation}
Since $\{T_1\  \cdots\  T_{n-1}\  N\}$ is an orthonormal basis, similar to \eqref{eq:ang_ut_rew}, we have
\begin{eqnarray} 
u_t  & = & \sum_{i,j} a^{T_i T_j} D_{T_i} D_{T_j} u  -\sum_i k_i a^{T_i T_i} D_N u + \sum_i a^{T_i N} D_{T_i} D_N u + \sum_i k_i a^{T_i N} D_{T_i} u \nonumber\\
& & + \sum_i a^{N T_i} D_N D_{T_i} u + a^{NN} D_N D_N u, 
\label{eq:4.12}
\end{eqnarray}
where we similarly define the following quantities:
\begin{equation}\label{eq:diri_aTN_def}
\begin{split}
    & a^{T_i T_j}  := \left( T_i^t A \ T_j \right),\quad
a^{T_i N}  := \left( T_i^t A \ N \right), \\
& a^{N T_i}  := \left( N^t A \ T_i \right),\quad
a^{NN}   := \left( N^t A \ N \right).
\end{split}
\end{equation}
Using the relation $D_{T_i} D_N u = k_i D_{T_i} u + D_{T_i} D_N u$ and \eqref{eq:4.12}, we can express $D_ND_Nu$ as
\begin{equation}
\begin{split}
    a^{NN} D_N D_N u\ =\ & u_t - \sum_{i,j} a^{T_i T_j} D_{T_i} D_{T_j} u + \sum_i k_i a^{T_i T_i} D_N u \\ 
    & - 2 \sum_i a^{T_i N} D_{T_i} D_N u - 2 \sum_i k_i a^{T_i N} D_{T_i} u.
\end{split}
\end{equation}
Upon expanding \eqref{eq:ell_max_norm}, we get $D_N u \ D_N D_N u + \sum_k D_{T_k} u \ D_N D_{T_k} u \leq 0$. Using the above expression for $D_ND_Nu$, we obtain,
\begin{multline*}
D_N u \Big( u_t - \sum_{i,j} a^{T_i T_j} D_{T_i} D_{T_j} u + \sum_i k_i a^{T_i T_i} D_N u  - 2 \sum_i a^{T_i N} D_{T_i} D_N u
    - 2 \sum_i k_i a^{T_i N} D_{T_i} u \Big)\\
    + a^{NN} \sum_l D_{T_l} u (D_{T_l} D_N u + k_l D_{T_l} u) \leq 0. 
\end{multline*}
Dividing the above by $|D_N u|$, we obtain,
\begin{equation} \label{eq:diri_ineq1}
\begin{split}
    \left(\sum_i k_i a^{T_i T_i}\right) |D_N u|\  \leq \ &|u_t| + \sum_{i,j} |a^{T_i T_j}| |D_{T_i} D_{T_j} u| \\ 
    & + 2 \sum_i |a^{T_i N}| \frac{\sum_l |D_{T_l} u| |D_{T_i} D_{T_l} u|}{|D_N u|} + 2 \sum_i |k_i| |a^{T_i N}| |D_{T_i} u| \\
    & + \sum_{i,l} |a^{NN}| \frac{|D_{T_l} u| |D_{T_i} u| |D_{T_l} D_{T_i} u|}{|D_N u|^2} + \sum_l |a^{NN}| |k_l| \frac{|D_{T_l} u|^2}{|D_N u|}
\end{split}
\end{equation}
Note that on $\partial \Omega$, we have
$|u_t|=|f_t|, \ |D_{T_i} u|=|D_{T_i} f|, \ |D_{T_i} D_{T_j} u|=|D_{T_i} D_{T_j} f|$
which according to \eqref{eq:diri_f_ass} are all bounded by the constant $M$.  

Next, without loss of generality, we assume $|D_N u| > \epsilon > 0$ for some constant 
$\epsilon$ which can be chosen as any positive number. Now we estimate the terms in 
\eqref{eq:diri_aTN_def}. The crucial estimate is the lower bound for $a^{T_iT_i}$:
\begin{equation}
\begin{split}
    a^{T_i T_i} & = G(\xi) \Bigg( \left( 1 - \frac{|D_{T_i} u|^2}{1+|Du|^2} \right) F(\xi) + 2 D_{T_i} u D^{T_i}_p F(\xi) + v^2 D^{T_i}_p D^{T_i}_p F(\xi) \Bigg) \\
    & \geq G(\xi) \Bigg( \left( 1 - \frac{M^2}{1 + M^2 + \epsilon^2} \right) F(\xi) - \frac{2 M m_1}{v} - m_2  \Bigg) \\
    & \geq g_0 \left(  \frac{1 + \epsilon^2}{1 + M^2 + \epsilon^2} m_0 - \frac{2 M m_1}{\sqrt{1+\epsilon^2}} - m_2 \right).
\end{split}
\label{eq:diri_aTT_low}
\end{equation}
By homogeneity for $F$ and $G$, upper bounds for $a^{T_iT_j}$ and the other terms are simple to obtain.
\begin{equation}\label{eq:diri_aTT_up}
    |a^{T_i T_i}| \leq G(\xi) \left( F(\xi) + \frac{2 M m_1}{v} + m_2  \right) \leq G_0 \left( M_0 + \frac{2 M m_1}{\sqrt{1+\epsilon^2}} + m_2 \right).
\end{equation}
For $i\neq j$, 
\begin{eqnarray*}
    |a^{T_i T_j}| & = & G(\xi) \Bigg| \left(  - \frac{D_{T_i} u \ D_{T_j} u}{1+|Du|^2} F(\xi) \right) +  D_{T_i} u \ D^{T_j}_p F(\xi) + D_{T_j} u \ D^{T_i}_p F(\xi) + v^2 D^{T_i}_p D^{T_j}_p F(\xi) \Bigg| \\
    & \leq & G(\xi) \left(  \frac{M^2}{1+M^2+\epsilon^2} F(\xi) + \frac{2 Mm_1}{v} + m_2 \right) \\
    & \leq & G_0 \left(  \frac{M^2}{1+M^2+\epsilon^2} M_0 + \frac{2 Mm_1}{\sqrt{1+\epsilon^2}} + m_2 \right),
\end{eqnarray*}

\begin{eqnarray*}
    |a^{T_i N}| &=&  G(\xi) \Bigg| \left(  - \frac{D_{T_i} u \ D_{N} u}{1+|Du|^2} F(\xi) \right) +  D_{T_i} u \ D^{N}_p F(\xi) + D_{N} u \ D^{T_i}_p F(\xi) + v^2 D^{T_i}_p D^{N}_p F(\xi) \Bigg| \\
    & \leq & G(\xi) \left( \frac{1}{2} F(\xi) + \frac{M m_1}{v} + m_1 + m_2 \right) \\
    & \leq & G_0 \left( \frac{1}{2} M_0 + (\frac{M}{\sqrt{1+\epsilon^2}} + 1) m_1 + m_2 \right),
\end{eqnarray*}

\begin{eqnarray*}
    |a^{NN}| &= & G(\xi) \left| \left( 1 - \frac{|D_N u|^2}{1+|Du|^2} \right) F(\xi) +  2 D_N u \ D^{N}_p F(\xi) + v^2 D^{N}_p D^{N}_p F(\xi) \right| \\
    & \leq & G(\xi) \left( \left( 1 - \frac{\epsilon^2}{1 + \epsilon^2 + (n-1)M^2} \right) F(\xi) + 2 m_1 + m_2  \right) \\
    & \leq &G_0 \left( \frac{1 + (n-1)M^2}{1 + \epsilon^2 + (n-1)M^2}  M_0 + 2 m_1 + m_2  \right).
\end{eqnarray*}
Hence the right hand side of \eqref{eq:diri_ineq1} is bounded by some constant $C(M,F,\epsilon,n)$. 

Now we simply take $\epsilon=1$. Then the right hand side of \eqref{eq:diri_ineq1} can be bounded by the boundary function $f$. Hence, we have the following estimate,
\begin{equation}    
\left(\sum_i k_i a^{T_i T_i}\right) |D_N u| \leq C(M,K,F,n).
\end{equation}
Then the lower and upper bounds \eqref{eq:diri_aTT_low}, \eqref{eq:diri_aTT_up}
for $a^{T_iT_i}$ become
\begin{eqnarray}
    a^{T_i T_i} &\geq& g_0 \left( \frac{2}{2+M^2} m_0 -\sqrt{2} M m_1 - m_2 \right) =: \gamma_1,\\
    |a^{T_i T_i}| &\leq& G_0 \left( M_0 +\sqrt{2} M m_1 +m_2 \right) =: \gamma_2.
\end{eqnarray}
From assumption \eqref{eq:diri_bd_cur_cond}, we have that 
$\sum_i k_i a^{T_i T_i}$ is bounded from below by a positive number that depends only
on the functions $F$ and $G$ and the boundary $\partial\Omega$ but not on the solution
$u$.
We can then conclude that 
\begin{equation*}
    |D_N u| \leq C(M, \partial\Omega, G, F, n).
\end{equation*}
finishing the proof of the gradient estimate.
\end{proof}

Next we consider the following elliptic problem: 
\begin{equation} 
\left\{ \begin{array}{cllcl}  \label{eq:diri_ell_pro}
     \lambda &=& G (Dw,-1) \ D^2_{p_i p_j} F (Dw,-1) w_{x_i x_j}   \quad & \text{in} & \ \Omega,  \\
     w &=& g \ \quad & \text{on} & \ \partial \Omega, 
\end{array} \right.
\end{equation}
for some constant $\lambda$.
Now this boundary function $g$ is independent of time, and satisfies
\begin{equation}
    |D_{T_i} g|,\ |D_{T_i T_j} g|   \leq M,
\end{equation} 
for some constant $M$. For the existence of solution of the above equation, we can in fact invoke the Fundamental Existence Theorem in \cite[Chapter 1, Section 1]{Se}. This is applicable in our present case because 
similar to the proof in Theorem \ref{thm:ell_grad_est}, by substituting $u_t$ with $\lambda$, we do have the gradient estimate for \eqref{eq:diri_ell_pro}, assuming the existence of a solution. 
Note        that the solution
$w$ exists for any finite value of $\lambda$. This is in contrast to the contact angle case in which $\lambda$ is uniquely determined by the function $\theta$ appearing in the boundary condition \eqref{eq:iso_ang_pro}.

Similar to the contact angle boundary condition, the function $w$ suggests a translating solution $\Tilde{w}(x,t) = w(x) + \lambda t$ to the parabolic problem 
\eqref{eq:diri_pro} with Dirichlet boundary function $f$ and initial condition $u_0(x)$ replaced by $g(x) + \lambda t$ and $u_0(x)=w(x)$, respectively.
With this, the next step is to study the asymptotic behavior of a solution to 
\begin{equation} \label{eq:diri_tran_pro}
    \left\{ \begin{array}{cllcl}
            u_t  &=& G(Du,-1) \  D^2_{p_i p_j} F (Du,-1) u_{x_i x_j}   \quad &  \text{in} & \ Q_T, \\
            u &=& g(\cdot) + \lambda t  \quad &  \text{on} & \ \Gamma_T,  \\
         u(\cdot,0) &=& u_0(\cdot)  \quad \ & \text{in} & \   \Omega_0.
    \end{array}  \right.
\end{equation}

With the gradient estimate in Theorem \ref{thm:ell_grad_est}, we can have the following simple but useful result.
\begin{corollary}
For a solution $u=u(x,t)$ of \eqref{eq:diri_ell_pro}, there exists a constant $C$ such that
\begin{equation}
    |u(x,t)-\lambda t| \leq C.
\end{equation}   
\end{corollary}
\begin{proof}
    This is simply due to the fact that $u(x,t)-\lambda t = g$ on the boundary and the gradient estimate obtained in Theorem \ref{thm:ell_grad_est}.
\end{proof}
Then we can follow the idea of \cite[Theorem 4.6]{DKY} to obtain the following
convergence result.
\begin{theorem}
Let $u(x,t)$ be a classical solution to \eqref{eq:diri_tran_pro}. Then there exists a constant $C$ such that $u(x,t)$ will converges to $w(x) + \lambda t + C$ as $t\longrightarrow \infty$. The convergence is uniform in $x\in\Omega$.
\end{theorem}

Just a word about the idea of proof. At time $t=0$, we can choose $C_1 < C_2$ such that
$w(x) + C_1 \leq u(x,0) \leq w(x) + C_2$ for all $x\in\Omega$. As both $w(x) + \lambda t$ and $u(x,t)$ are solutions of \eqref{eq:diri_tran_pro}, we have
$w(x) + \lambda t + C_1 \leq u(x,t) \leq w(x) + \lambda t + C_2$ for all $t > 0$.
Then by strong maximum principle, we can choose a $C_1'$ and $C_2'$ satisfying
$C_1 < C_1' < C_2' < C_2$ such that 
$w(x) + \lambda t + C_1' \leq u(x,t) \leq w(x) + \lambda t + C_2'$ for $t > t'$
where $t' > 0$. Iterating this procedure we can in fact keep increasing $C_1'$ and decreasing $C_2'$ such that they converge to a single constant. With more careful estimation, we can in fact show that the convergence is exponential in time, as in
\cite[Theorem 4.6]{DKY}.

\bibliographystyle{plain}
\bibliography{references}

\end{document}